\documentclass{amsart}

\usepackage{graphicx}
\usepackage{enumerate,amssymb,amstext,amsmath,hyperref,amsfonts,amscd,amsbsy,amsthm}

\def\ga{\alpha}

\def\gc{\gamma}

\def\th{\theta}

\def\gw{\omega}

\def\bgw{\boldsymbol{\gw}}

\def\gC{\Gamma}

\def\gW{\Omega}

\def\br{{\bf r}}
\def\bs{{\bf s}}

\def\bv{{\bf v}}

\def\bx{{\bf x}}
\def\by{{\bf y}}

\def\b0{{\bf0}}

\def\x{\times}

\def\<{\langle}
\def\>{\rangle}
\def\bbR{{\mathbb R}}

\def\Lse{\mathfrak{se}}

\def\Lso{\mathfrak{so}}

\def\Lt{\mathfrak{t}}

\def\({\left(}
\def\){\right)}

\def \tete{t\^{e}te-a-t\^{e}te}

\newtheorem{theorem}{Theorem}[section]

\theoremstyle{definition}
\newtheorem{definition}[theorem]{Definition}

\begin{document}

\title{Polynomial Invariants and SAGBI Bases for Multi-screws}

\author{Deborah Crook \and Peter Donelan}
\address{School of Mathematics \& Statistics, Victoria University of Wellington, PO Box 600, Wellington 6140, New Zealand}
        \email{peter.donelan@vuw.ac.nz}

\thanks{This work was supported in part by a Victoria University of Wellington Master's Scholarship}

\maketitle

\begin{abstract}

Polynomial invariants for robot manipulators and their joints arise from the adjoint action of the Euclidean group on its Lie algebra, the space of infinitesimal twists or screws.  The aim of this paper is to determine basic sets of generating polynomials for multiple screws.  Techniques from the theory of SAGBI bases are introduced. As a result,  a complete description is provided of the polynomial invariants for screw pairs and some results for screw triples are obtained.  The invariants are shown to be related to Denavit--Hartenberg parameters.
\end{abstract}


\section{Introduction}
\label{s:intro}
The infinitesimal motion of a rigid body in space is classically described by a screw, generically a helical motion about an axis somewhere in space, special cases being pure rotation (not translation along the axis) and pure translation.  There is a natural equivalence between screws arising from the action of the Euclidean group and a fundamental characteristic of a screw is its {\em pitch}, $p$. This is the ratio of two degree~2 homogeneous polynomials in the Pl\"ucker coordinates $(\bgw,\bv)$, the Klein form $\bgw.\bv$ and the Killing form $\bgw.\bgw$. These forms are themselves invariants of the adjoint action of the Euclidean group.  Moreover, they generate all such invariant polynomials, so are fundamental.  

A serial mechanism or manipulator (SM) consists of a sequence of rigid bodies connected by 1~degree-of-freedom (dof) joints that are, in principle any of revolute (R), prismatic (P) or helical (H). In practice, H joints are rarely used~\cite{hunt}.  Each joint, in a given configuration of the SM can be identified by a screw $X$, where $X$ is type R when $p=0$ (pure rotation) and type $P$ when $p=\infty$, ie $\bgw=0$ (pure translation). Type H corresponds to finite non-zero $p$. Therefore a crude classification of SMs simply requires identification of the sequence of joints.  However, that in itself is not sufficient to reconstruct an SM; one must know additional design parameters that describe the relative placement of the joints in successive components. This is typically done by means of Denavit--Hartenberg (DH) parameters~\cite{den}.  Denavit and Hartenberg introduced a matrix product notation for representing the motion arising from SMs.  Brockett subsequently showed how to rewrite this, in a purer form,  as a product of exponentials~\cite{broc}.  

Nevertheless, although DH parameters are valuable for describing SMs they do not have a clear theoretical basis.  
A mathematically more natural approach to classifying and identifying SMs is in terms of fundamental invariants for the screw sequences that appear in the product-of-exponentials form.  Towards this end, one would like to identify invariants for pairs, triples and general multiples of screws.  These are generally referred to as {\em vector invariants}.

Invariant polynomials of a group action form a subring in the ring of polynomials, where the variables are coordinates for the space acted upon.  One hopes to find generators for the invariant ring---a set of polynomials in terms of which every other invariant can be written.  It is known that for a large class of groups, the {\em reductive groups}, the invariant ring is finitely generated.  On the other hand there are non-reductive counter-examples to finite generation.  For a given group action, the First Fundamental Theorem of Invariant Theory asserts the polynomial invariants are finitely generated and gives a list of generators.  The list may not be algebraically independent: relations between the generators are called {\em syzygies} and the Second Fundamental Theorem (when it holds) asserts that the syzygies themselves are finitely generated. From this it is possible to obtain a reasonably precise description of the space of orbits of the action as an algebraic variety.  Further, a theorem proved by Hochster and Roberts asserts that the ring of invariants has a property known as Cohen--Macaulay~\cite{hoch}. This entails that every invariant polynomial can be written in the form
\[
f=\sum_{i=1}^t f_i(\th_1,\dots,\th_n)\eta_i,
\]
where $\th_1,\dots,\th_n$ is a set of {\em primary} invariant polynomials, $\eta_1,\dots,\eta_t$ a set of {\em secondary} invariants with $\eta_1=1$ and $f_1,\dots,f_n$ themselves polynomials. 

The Euclidean group is however non-reductive; nevertheless, its structure as a semi-direct product over a the reductive group $SO(3)$ does provide some hope that its invariants are finitely generated.   Known results mostly concern the standard action of the Euclidean group~(\ref{e:std}).  Weyl describes the vector invariant theory for the special orthogonal groups~\cite{weyl}. The fundamental theorems for the standard action of the Euclidean group are proved in~\cite{dalb}. Panyushev uses deep results from invariant theory and algebraic geometry to establish generators for invariants of semi-direct products~\cite{pan}, including the Euclidean group, in special cases and these encompass the case of screw pairs. Selig~\cite{sel} establishes a number of invariants for screw systems, that is subspaces of Lie algebra. In fact, his approach uses a basis of twists for a screw system and hence the resulting invariants coincide with those for multi-screws.  Connections with the classical theory of invariants and line geometry are established in~\cite{syz} and form a branch of this broad area of study.

In this paper, a more computational approach is employed, which it may be possible to extend to multi-screws. 
Computational results in invariant theory originate with Young's straightening law~\cite{you}, a procedure for reducing an invariant to a normal form.  More recently, Gr\"obner bases have provided a more general approach for computation in polynomial ideals~\cite{cox}.  As the invariant polynomials form a subalgebra rather than an ideal, the analogous SAGBI (or canonical subalgebra) basis theory is relevant~\cite{kap,robb}.  SAGBI bases provide algorithms for testing whether a given polynomial is invariant and of reducing it in terms of the basis.  Using these methods a generating set for the polynomial invariants on screw pairs is obtained, together with a list of invariants for screw triples and a conjecture regarding a generating set for screw triples.  

A longer term goal of this research is to obtain a complete understanding of the polynomial invariants  for multiple screws of any number, both as a set and in sequence in the way they occur in a serial manipulator.  The latter problem involves the additional subtlety that the multi-screw changes as the manipulator moves through different configurations and, moreover, since the manipulator may have singular configurations, these may become linearly dependent so that even the dimension of the screw system is not invariant.  In these cases there is an additional action of copies of the real numbers that propagate through the product of exponentials.  

In Section~\ref{s:adjoint}, the Euclidean group is defined together with its adjoint action on its Lie algebra, the space of twists or screws.  A brief introduction to polynomial invariants (Section~\ref{s:invariant}) and SAGBI bases (Section~\ref{s:sagbi}) follows.  The techniques are used in the computation of invariants and theorems on generating sets are established in Section~\ref{s:comp}. The connection with Denavit--Hartenberg parameters is described in Section~\ref{s:dhinv}.

\section{Adjoint action of the Euclidean group}
\label{s:adjoint}
The displacement of a rigid body in ordinary Euclidean 3-dimensional space $E^3$ is described by an element of the special Euclidean group $SE(3)$, a 6-dimensional Lie group.  Given a choice of origin and orthonormal coordinates, $E^3$ may be regarded as the vector space $\bbR^3$ with the Euclidean inner product.  With respect to these coordinates, a displacement in $SE(3)$ can be described by means of a combination of rotation about the origin, represented by a $3\x3$ orientation-preserving orthogonal matrix $R\in SO(3)$ (the special orthogonal group) and a translation $\br\in\bbR^3$.  In this form, composition of displacements is not by direct product of the rotation and translation subgroups, rather it is a semi-direct product $SE(3)\cong SO(3)\ltimes\bbR^3$, with composition:
\[
(R_2,\br_1)\cdot(R_1,\br_1)=(R_2R_1,R_2\br_1+\br_2),
\]
and only the translations form a normal subgroup. 
A displacement $A=(R,\br)$ acts on a point  $\bx=(x_1,x_2,x_3)^t\in\bbR^3$ by 
\begin{equation}
\label{e:std}
(R,\br).\bx=R\bx+\br.
\end{equation}
The motion of a rigid body is a path $(R(t),\br(t))$ where $t\in\bbR$ denotes a time parameter. Assuming the path to be differentiable and that $(R(0),\br(0))=(I,\b0)$, the identity displacement in $SO(3)\ltimes\bbR^3$, then the derivative $\bs=(\dot{R}(0),\dot{\br}(0))=(\gW,\bv)$ belongs to the tangent space to $SE(3)$ at the identity, that is its Lie algebra $\Lse(3)$. Elements of $\Lse(3)$ are called {\em twists} and are closely related to the one-dimensional subspaces spanned by non-zero twists, called {\em screws}. In given coordinates,  $\Lse(3)$ inherits a semi-direct sum structure, $\Lso(3)\oplus\Lt(3)$, where $\gW\in\Lso(3)$ is a $3\x3$ skew-symmetric matrix 
\[
\gW=\begin{pmatrix}0&-\gw_3&\gw_2\\ \gw_3&0&-\gw_1\\-\gw_2&\gw_1&0\end{pmatrix}
\]
which it is  convenient to identify with the 3-vector $\bgw=(\gw_1,\gw_2,\gw_3)^t$. The infinitesimal translations $\Lt^3$ can be written simply as 3-vectors $\bv\in\bbR^3$.  The components of the 6-vector $(\bgw^t,\bv^t)^t$ are referred to as the {\em Pl\"ucker coordinates} of a twist.  For convenience these will be denoted $(\bgw,\bv)$. 

A change of coordinates in $E^3$ can be represented by a transformation $T\in SO(3)\ltimes\bbR^3$ and a Euclidean displacement given by $A$ transforms under conjugation to $TAT^{-1}$.  Differentiating the conjugation of a path $A(t)$ through the identity gives rise to the {\em adjoint action} of $SE(3)$ on its Lie algebra $\Lse(3)$:
\begin{equation}
\label{e:adj}
(R,\br).(\gW,\bv)=(R\gW R^{-1}, R\gW R^{-1}\br+R\bv).
\end{equation}
Since the Lie algebra is a 6-dimensional vector space, a better way of writing the adjoint action is in terms of the Pl\"ucker coordinates, in which case
\[
(R,\br),(\bgw,\bv)=(R\bgw,R\bgw\x \br+R\bv)
\]
or better still, by replacing $\br=(t_1,t_2,t_3)^t$ by the skew-symmetric matrix $T=\begin{pmatrix} 0&-t_3&t_2\\t_3&0&-t_1\\-t_2&t_1&0\end{pmatrix}$, the adjoint action is given by the $6\x6$ representation~\cite{sel}:
\begin{equation}
\label{e:adjrep}
 \begin{pmatrix}
  R&0\\
  TR&R\\
 \end{pmatrix}
\begin{pmatrix}
 \bgw\\
 \textbf{v}\\
\end{pmatrix}=
\begin{pmatrix}
 R\bgw\\
 TR\bgw + R\textbf{v}\\
\end{pmatrix}
\end{equation}
One can clearly extend the adjoint action to a vector of screws or {\em multi-screw} $((\gw_1,\bv_1),\dots,(\bgw_k,\bv_k))$.  It is worth noting the two subactions associated with the rotation and translation subgroups.  These are represented by matrices of the form
\begin{equation}
\label{e:subaction}
\begin{pmatrix}
  R&0\\
  0&R\\
 \end{pmatrix}\quad\mbox{and}\quad
 \begin{pmatrix}
  I&0\\
  T&I\\
 \end{pmatrix},
\end{equation}
respectively, where $I$ is $3\x3$ the identity matrix. In particular, the rotation action is the double of the standard action of $SO(3)$ so that its action on a $k$-screw corresponds to the $2k$-vector action of the special orthogonal group, whose invariant theory is  fully described in~\cite{weyl,rich}.

\section{Invariant polynomials}
\label{s:invariant}
Given the action of a group $G$ on an affine space $K^n$ (or a variety contained in it) there is an induced action of the group  on the ring of polynomials  $K[x_1,\dots,x_n]$ given by $(Af)\bx)=f(A.\bx)$, and the {\em invariant polynomials} are those for which $Af=f$. The basic theorems for the classical groups, where both First and Second Fundamental Theorems hold, are established in~\cite{weyl}. In particular, the vector invariants of $SO(3)$ acting on a set of $m$ vectors $\bx_1,\dots,\bx_m\in\bbR^3$ are generated by the following polynomials:
\begin{equation}
\label{e:so3gens}
\begin{split}
&\bx_i.\bx_j, \quad1\leq i,j\leq m; \\
&[\bx_{i},\bx_{j},\bx_{k}],\quad 1\leq i<j<k\leq m,
\end{split}
\end{equation}
where the bracket denotes the determinant of the $3\x3$ matrix whose columns are the vectors and is defined only when $m\geq3$. Further, there are three families of relations generating the syzygies, but again these only arise when $m\geq 3$. In particular, let $X$ denote the $m\x m$ matrix whose $(i,j)^{th}$ entry is the scalar product $\bx_i\cdot\bx_j$.  Denote its $k\x k$ minors by: 
\begin{equation}
\label{e:minors}
f^{i_1,\ldots,i_{k}}_{j_1,\ldots,j_{k}} := \begin{vmatrix} \bx_{i_1}\cdot\bx_{j_1}&\ldots&\bx_{i_1}\cdot\bx_{j_{k}}\\ \vdots&\ddots&\vdots\\ \bx_{i_{k}}\cdot\bx_{j_1}&\ldots&\bx_{i_{k}}\cdot\bx_{j_{k}}\\ \end{vmatrix}
\end{equation}
Then for $m\geq 4$, the relations $f^{i_1,\ldots,i_{4}}_{j_1,\ldots,j_{4}} =0$ form one family of syzygies.

Taking the Pl\"ucker coordinates as variables on the Lie algebra, there is an induced action of the Euclidean group $SE(3)$ on the ring of polynomials  $\bbR[\bgw,\bv]=\bbR[\gw_1,\gw_2,\gw_3,v_1,v_2,v_3]$, It is well known that the ring of such polynomials, denoted $\bbR[\bgw,\bv]^{SE(3)}$, contains the Klein and Killing forms, $\bgw.\bv$ and $\|\bgw\|^2=\bgw.\bgw$, respectively.  It is, in fact generated by these, that is to say, there is an isomorphism:
\begin{equation}
\label{e:eucgen}
\bbR[\bgw,\bv]^{SE(3)}\cong\bbR[\bgw.\bv,\bgw.\bgw],
\end{equation}
and there are no relations between these invariants.  The proof in~\cite{don}, where the general case for $SE(n)$ is proved, makes use of reduction to a maximal torus and then careful analysis of how the invariants there lift back to the whole group.  

Another approach is found in~\cite{pan} where a procedure is introduced to find explicit generators for the adjoint action of a semi-direct product $G\ltimes V$, where $G$ is a group whose adjoint action is isomorphic to its action on $V$. This is the case for $G=SO(3)$ with $V=\bbR^3$. It is shown that if  $K[V]^G=K[f_1,\dots, f_k]$, there are no non-trivial syzygies and certain other technical conditions hold, then $K[V\x V]^{G\ltimes V}$ is (freely) generated by $f_1,\dots,f_k$ together with $d(f_i)_\bx.\by$ where $\bx,\by$ denote coordinates on the two copies of $V$ respectively and $df_\bx$ is the derivative vector of polynomials of $f$. Moreover, when the hypotheses hold, the invariants of the ``translation" subgroup $1\ltimes V$ are given as the coordinates of the first factor $\bx$ together with the new invariants above. These results are re-established below by a different method in the cases $m=1,2$ for $SE(3)=SO(3)\ltimes\bbR^3$.

\section{SAGBI bases}
\label{s:sagbi}
Computations in an ideal $I$ of a polynomial ring $K[x_1,\dots,x_n]$ (for some field of coefficients $K$) are effectively done using a Gr\"obner basis $G$. Given an ordering of the  terms in the variables (that satisfies certain natural properties---see for example~\cite{cox}), $G$ is a  generating sets for $I$ such that every polynomial in $I$ has leading term divisible by an element of $G$. Using a Gr\"obner basis provides an effective way of rewriting any element of $I$ in terms of the generators, i.e. $\sum_i u_ig_i$, where $u_i\in K[x_1,\dots,x_n]$, $g_i\in G$. However, the invariant polynomials form a closed set only under addition and multiplication by each other, not by general polynomials in the ring, so they form only a subalgebra, not an ideal.  The theory of SAGBI bases \cite{kap,robb} is an analogue for subalgebras $A$ instead of ideals.  

\begin{definition}
\label{d:sagbi}
$S$ is a {\em SAGBI basis} for a subalgebra $A$, with respect to a given term ordering, if the algebra of leading terms of polynomials in $S$ generates the subalgebra of all leading terms for $A$, that is, if $f\in A$ then its leading monomial  is a product of leading monomials of elements in $S$. 
\end{definition} 

The acronym SAGBI stands for ``subalgebra analogue of Gr\"obner bases for ideals".  There is a procedure, similar to Buchberger's algorithm for Gr\"obner bases, for building a SAGBI basis from a given generating set.  This uses a SAGBI reduction process, termed {\em subduction},the subalgebra version of the Division Algorithm, and analogues of $S$-polynomials, termed {\em {\tete}s}.  However,  given that subalgebras are not necessarily finitely generated, in the sense of (\ref{e:eucgen}), the procedure is not genuinely algorithmic. Its output may be infinite and/or it may not terminate in a finite number of steps. Indeed, even when $A$ itself is know to be finitely generated, there may not exist a finite SAGBI basis with respect to some or all term orderings.  The reader is referred to~\cite{robb} for details of the algorithm and its shortcomings and to~\cite{brav} for a more recent discussion of the existence of finite bases. 

The application of the procedure to invariant subalgebras relies on the following construction~\cite{vasc}. Suppose $\gC$ is a group defined as an algebraic variety in some affine space $K^m$, that is, $\gC$ is defined by means of polynomial equations on $K^m$ that generate an ideal $I(G)\subseteq K[a_1,\c\dots,a_m]$. Let $\gC$ act on a variety  $X$ in the affine space $K^n$, whose defining equations determine the ideal $I(X)\subseteq K[x_1,\dots,x_n]$.  The action $\psi:\gC\x X\to X$ gives rise to polynomial functions $\psi^*(x_i)\in K[a_1,\dots,a_m,x_1,\dots,x_n]$, $i=1,\dots,m$, being the components of $\gc.\bx$ in $K^m$. It is convenient to write $y_i=\psi^*(x_i)$ and treat the polynomials when appropriate as variables of a polynomial ring in their own right. Then:
\begin{equation}
\label{e:elim}
\begin{split}
K[x_1,\dots,x_m]^\gC&=K[y_1,\ldots,y_n]\cap \frac{K[x_1,\ldots,x_n]}{I(X)}\\
       &\subseteq \frac{K[a_1,\ldots,a_m,x_1,\ldots,x_n]}{I(G)+I(X)}
\end{split}
\end{equation}
This translates to a statement about SAGBI bases.  A term ordering in which any term that involves one or more of a given subset of variables is if higher order than any that contain none of these variables is called an {\em elimination order}.  A pure lexicographic order certainly satisfies this condition for any leading subset of the variables. The following theorem is proved by Stillman and Tsai~\cite{stil}.

\begin{theorem}
\label{t:sagbi}
With the notation above, let $\{f_1,\ldots,f_r\}$ be a SAGBI basis for $K[y_1,\ldots,y_n]$ with respect to an elimination order, in which any polynomial involving $a_i$, $1\leq i \leq m$ exceeds those in $x_j$, $1\leq j\leq n$, alone. Then $\{f_1,\ldots,f_r\}\cap K[x_1,\ldots,x_n]/I(X)$ is a SAGBI basis for $K[X]^G$, and hence a generating set.
\end{theorem}

\section{Computing Invariants}
\label{s:comp}
Richman~\cite{rich} proves the general $n$-dimensional case of the following result for $SO(3)$:
\begin{theorem}
\label{t:sosagbi}
A SAGBI basis for the invariants of $SO(3)$, under lexicographic ordering $x_{11} > x_{12}>\ldots >  \ldots > x_{m3}$ is given by the minors $f^{i_1,\ldots,i_k}_{j_1,\ldots,j_k}$ for $k=1,2$, together with the determinants $|\textbf{x}_{i_1}\ldots\textbf{x}_{i_3}|$ for all $1\leq i_1 < i_2 < i_3 \leq m$.
\end{theorem}
In other words, in addition to the generators (\ref{e:so3gens}) the $2\x2$ minors of the matrix of scalar products are required. 

\subsection{Screw invariants}
As an example of the application of this theorem, consider the translation sub-action of the adjoint action of $SE(3)$ in (\ref{e:subaction}). Lexicographic ordering with $t_1 > t_2 > t_3 > \gw_1 > \gw_2 > \gw_3>v_1 > v_2 > v_3$, is used.  The first three variables relate to the group $G=\bbR^3$ and the last six to the Lie algebra $X=\Lse(3)$.  Since both of these are already affine spaces the ideals $I(G)$ and $I(X)$ are trivial. The components of the action are given, in decreasing term order, by:
\begin{equation}
 \begin{split}
y_1 &= \omega_1, \quad y_2 = \omega_2,\quad y_3 = \omega_3\\
y_4 &= t_2 \omega_3 - t_3 \omega_2 + v_1\\
y_5 &= -t_1 \omega_3 + t_3 \omega_1 + v_2\\
y_6 &= t_1 \omega_2 - t_2 \omega_1 + v_3\\
\end{split}
\end{equation}
Of course, $y_1,y_2,y_3$ are already invariants. Application of the SAGBI basis construction procedure yields one additional polynomial, $f_1=\omega_1 v_1 + \omega_2 v_2 + \omega_3 v_3=\bgw.\bv$. Further details of the computation can be found in~\cite{croo}.  Hence, by Theorem~\ref{t:sagbi}, a SAGBI basis for $\bbR[\gw_1, \gw_2, \gw_3,v_1, v_2, v_3]^{\bbR^3}$, the invariant polynomial subalgebra,  is $\{\gw_1,\gw_2,\gw_3,\bgw.\bv\}$.  This can be used to provide an alternative derivation of the finite generation theorem for the adjoint action of $SE(3)$. 

\begin{theorem}
\label{t:screwinv}
The invariant subring $\bbR[\omega_1,\ldots,v_3]$ for the adjoint action of $SE(3)$  has a SAGBI basis given by $\{\boldsymbol{\omega}\cdot\boldsymbol{\omega},\boldsymbol{\omega}\cdot\textbf{v}\}$.
\end{theorem}

\begin{proof}
Let $f$ be an $SE(3)$ invariant for the action on a single screw. As it is, in particular, a translational invariant, it can be written via the SAGBI basis in the form:
\begin{equation}
 f = \sum_{b} g_b.(\boldsymbol{\omega}\cdot\textbf{v})^b,
\end{equation}
where $g\in \bbR[\omega_1,\omega_2,\omega_3]$ and the sum is finite. Equally, $f$ must also be a rotational invariant. If $A = \begin{pmatrix} R\quad 0\\0\quad R\\ \end{pmatrix}\in SE(3)$, where $R\in SO(3)$, then $A(\boldsymbol{\omega},\textbf{v})=R\bgw.R\bv=\bgw.\bv$. Hence
\begin{equation}
 Af = \sum_{b} (Ag_b)(\omega_1,\omega_2,\omega_3)(\boldsymbol{\omega}\cdot \textbf{v})^b= f.
\end{equation}
This entails $Ag_b = g_b$ for all $b$. Hence $g_b$ is an invariant of the standard action of $SO(3)$ on $\boldsymbol{\omega}$, for which a SAGBI basis is $\{\boldsymbol{\omega}\cdot\boldsymbol{\omega}\}$, by Theorem~\ref{t:sosagbi}.
Hence $g_b \in \bbR[\boldsymbol{\omega}\cdot\boldsymbol{\omega}]$ and so the $SE(3)$ invariants are generated by $\{\boldsymbol{\omega} \cdot \boldsymbol{\omega}, \boldsymbol{\omega} \cdot \textbf{v}\}$. Moreover, as there are no cancellations among leading terms of these generators it is also a SAGBI basis.
\end{proof}

Unfortunately, direct application of Theorem~\ref{t:sagbi} does not appear to yield a SAGBI basis for the rotational sub-action. On the other hand, there are already known SAGBI bases from Theorem~\ref{t:sosagbi}. 

\subsection{Screw pairs and triples}
The translation group approach may be used in the case of $2$-vector invariants of adjoint action of $SE(3)$. Consider a screw-pair $(\boldsymbol{\omega}_i,\textbf{v}_i)$,  with $\boldsymbol{\omega}_i=(\omega_{i1},\omega_{i2},\omega_{i3})$, $\bv_i=(v_{i1},v_{i2},v_{i3})$ for $i=1,2$.  We adopt  lexicographic order with
\begin{align*}
&t_1>t_2>t_3>\omega_{11}>\omega_{12}>\omega_{13}>\omega_{21}>\omega_{22}>\omega_{23}\\
&\kern6cm>v_{11}>v_{12}>v_{13}>v_{21}>v_{22}>v_{23}.
\end{align*}
Computation gives a SAGBI basis for the translation sub-action as follows:
\begin{align}
\label{sagbi2}
 &\omega_{in}, \;\quad\qquad1\leq i \leq 2,\quad 1\leq n \leq 3\\
 &\boldsymbol{\omega}_i\cdot \textbf{v}_i,  \qquad1\leq i \leq 2&\notag\\
 &\boldsymbol{\omega}_1\cdot \textbf{v}_2 + \boldsymbol{\omega}_2 \cdot \textbf{v}_1\notag\\
 &w_{11}(w_{22}v_{22}+w_{23}v_{23})\notag\\
 &\qquad-w_{21}(w_{12}v_{22}+w_{21}v_{23}+w_{21}v_{11}+w_{22}v_{12}+w_{23}v_{13})\notag
 \end{align}
The last translation invariant derives from the \tete
\begin{equation}
\label{last_sagbi}
\gw_{11}(\boldsymbol{\omega}_2\cdot \textbf{v}_2)-\gw_{21}(\boldsymbol{\omega}_1\cdot \textbf{v}_2 + \boldsymbol{\omega}_2 \cdot \textbf{v}_1)
\end{equation}
and, together with the $\omega_{in}$, is not a rotation invariant. That the set forms a SAGBI basis can be verified by a computer algebra system such as Macaulay2~\cite{macaulay}.

\begin{theorem}
\label{t:2screws}
The invariant ring of $SE(3)$ acting on two screws is generated by 
\begin{equation}
\label{e:2screwinv}
\{\boldsymbol{\omega}_i\cdot\boldsymbol{\omega}_j,\boldsymbol{\omega}_i\cdot \textbf{v}_i,\boldsymbol{\omega}_1\cdot \textbf{v}_2 + \boldsymbol{\omega}_2\cdot \textbf{v}_1\,|\,1\leq i,j \leq 2\}.
\end{equation}
\end{theorem}

\begin{proof}
This is  similar to that for Theorem~\ref{t:screwinv}. Any $SE(3)$ invariant $f(\boldsymbol{\omega}_1,\bv_1,\boldsymbol{\omega}_2,\bv_2)$ is also a translation invariant so, using the SAGBI basis, can be written in the form
\[
\sum_{a,b,c} g_{a,b,c}(\boldsymbol{\omega}_1,\boldsymbol{\omega}_2)f_1^af_2^bf_3^c,
\]
where $f_1,f_2,f_3$ are the middle three invariants in the SAGBI basis~(\ref{sagbi2}). This is possible because of~(\ref{last_sagbi}).
Since the $f_i$ are also rotational invariants, it follows that the coefficients $g_{a,b,c}$ must also be and the result follows from Theorem~\ref{t:sosagbi}.  We note that these are the generators of Panyushev~\cite{pan}, described in Section~\ref{s:invariant}, and of Selig~\cite{sel} relating to 2-systems. \end{proof}

In the case of screw triples $(\bgw_i,\bv_i)$, $i=1,2,3$, application of Theorem~\ref{t:sagbi} has yielded the translational invariants:
\begin{align}
 &\omega_{in}\qquad\qquad\quad\: &1\leq i\leq 3,\quad1\leq n\leq 3\notag\\
 &\boldsymbol{\omega}_i\cdot \textbf{v}_i \qquad\qquad &1\leq i \leq 3\notag\\
 &\boldsymbol{\omega}_i\cdot \textbf{v}_j + \boldsymbol{\omega}_j \cdot \textbf{v}_i&1\leq i<j\leq 3\notag\\
 &z_{123},\:\; z_{231},\:\; z_{312}&\notag\\
 &z_{121}-z_{323},\:\; z_{232}-z_{131},\:\; z_{313}-z_{212}\kern-3cm&
\end{align}
where $z_{ijk} = \begin{vmatrix} \omega_{1i}\quad \omega_{2i}\quad \omega_{3i}\\\omega_{1j}\quad \omega_{2j}\quad \omega_{3j}\\ v_{1k}\quad v_{2k}\quad v_{3k}\\ \end{vmatrix}$. However, owing to the length of the computation, it is not clear that the computation has terminated.  

In the case of $z_{iji}-z_{kjk}$, neither of the terms $z_{iji}$, $z_{kjk}$ alone is a translational invariant. Furthermore, none of the final six invariants is a rotational invariant, and so the method used for the cases $m=1,2$ to determine $SE(3)$ invariants (assuming that we have a full listing of translational invariants) will not work. By comparison with the known set of $SO(3)$ invariants, however, the $SE(3)$ invariant  $z_{123} + z_{231} + z_{312}$ can be identified. This is equal to the sum 
\[
[\bv_1,\bgw_2 ,\bgw_3] + [\bgw_1, \bv_2,\bgw_3]+[\bgw_1, \bgw_2, \bv_3] 
\]
of rotation invariants. We conjecture that the following invariants generate the ring of Euclidean invariants of three screws:
\begin{align}
\label{e:3screwinv}
&\boldsymbol{\omega}_i\cdot\boldsymbol{\omega}_j,&1\leq i,j \leq 3\notag\\
&\boldsymbol{\omega}_i\cdot \textbf{v}_i,&1\leq i \leq 3\notag\\
&\boldsymbol{\omega}_i\cdot \textbf{v}_j + \boldsymbol{\omega}_j\cdot \textbf{v}_i,&1\leq i<j \leq 3\notag\\
&[\bgw_1,\bgw_2 ,\bgw_3]\notag\\
&[\bv_1,\bgw_2 ,\bgw_3] + [\bgw_1, \bv_2,\bgw_3]+[\bgw_1, \bgw_2, \bv_3] \kern-2cm
\end{align}

\section{Denavit--Hartenberg Parameters}
\label{s:dhinv}
The DH parameters for a series of screws (or, more properly, twists so that there is a definitive direction to the axis corresponding to a positive angle of rotation) are as illustrated in Figure~\ref{f:dh}.
   \begin{figure}[thpb]
      \centering
      \includegraphics[scale=0.4]{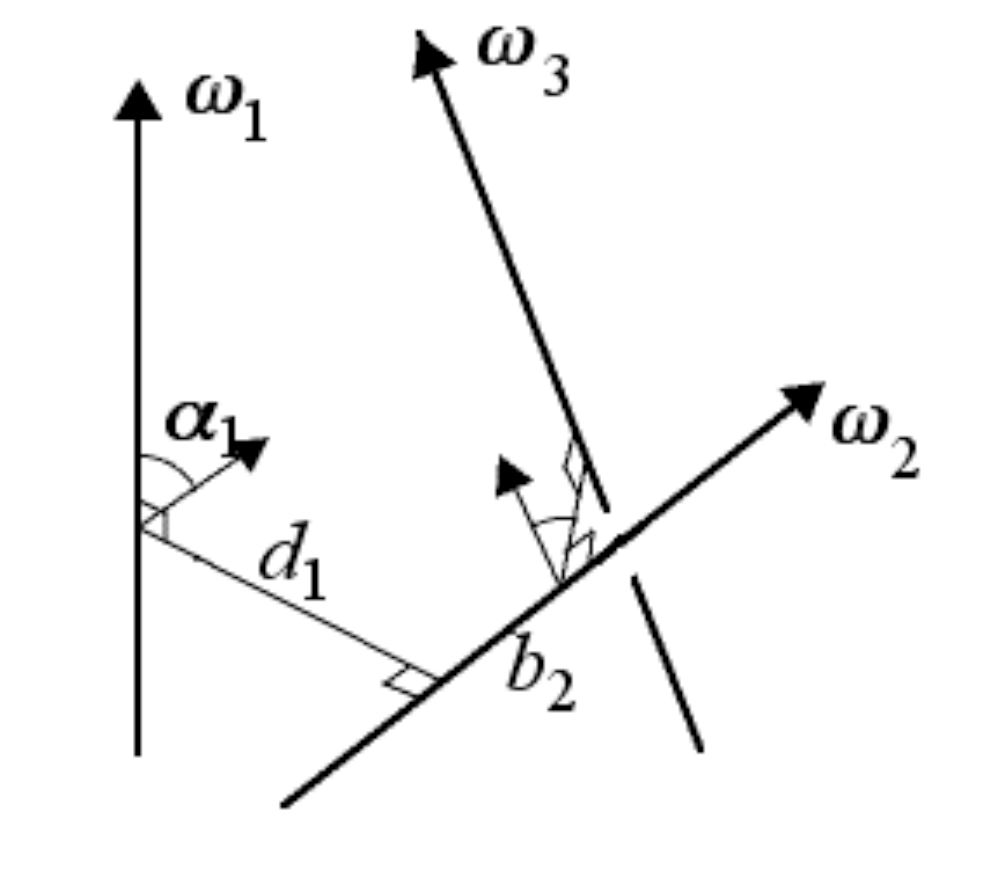}
      \caption{Denavit--Hartenberg parameters for a screw triple}
      \label{f:dh}
   \end{figure}
The parameters are the {\em twist angle} $\ga_i$ between the direction vectors of two consecutive screws $\bgw_i$, $\bgw_{i+1}$ , the {\em displacement} $d_i$ being the length of the common perpendicular between the screw axes and the {\em offset} $b_i$, which is only defined when there are three consecutive screws and is the distance along the middle one between the feet of the common perpendicular to the other two.  Note that if one is simply interested in three screws, not necessarily connected in any order, then there is an offset associated with every triple.  It is clear that the DH parameters are not polynomials in the Pl\"ucker coordinates. Nevertheless they are related, at least in the case of twist angle and displacement.  For a pair of screws $(\bgw_i,\bv_i)$, $i=1,2$ with twist angle $\ga$ and displacement $d$: 
\begin{equation}
\label{e:dhinv}
\begin{split}
\cos\ga&=\frac{\bgw_1\cdot\bgw_2}{\sqrt{(\bgw_1\cdot\bgw_1)(\bgw_2\cdot\bgw_2)}}\\
d\sin\ga&=\frac{\bgw_1\cdot\bv_2+\bgw_2\cdot\bv_1}{\sqrt{(\bgw_1\cdot\bgw_1)(\bgw_2\cdot\bgw_2)}}
\end{split}
\end{equation}
The expressions on the right of (\ref{e:dhinv}) are written in terms of the fundamental invariants for screw pairs in Theorem~\ref{t:2screws}.  A version of these formulae can be found, for example, in~\cite{mas}. 

In the case of the offset for a screw triple, it is certainly possible to obtain an expression in terms of the Pl\"ucker coordinates of the screws and to show that this expression is an $SE(3)$ invariant.  However, so far, it has not been possible to write this expression in terms of the conjectured fundamental generators for the invariant polynomials.  

\section{Conclusion}

Euclidean invariants have long played a fundamental role in rigid body kinematics.  The expression of quantities such as pitch and the DH parameters in terms of polynomial invariants demonstrates that knowledge of those particular invariants is equally important. From the mathematical point of view, there is a rich theory for polynomial invariants, both in the classical literature, where the central results are the Fundamental Theorems of Invariant Theory, and in the modern computational invariant theory via Gr\"obner and SAGBI bases.  It is surprising that relatively little has been known until recently about the adjoint invariants of the Euclidean group.  This paper has briefly surveyed earlier results and has introduced techniques of SAGBI bases to the search for more detailed knowledge of Euclidean invariants.  While these have been only partially successful, enumeration of SAGBI bases in some important cases suggest that it can be a powerful tool, leading to a deeper understanding of the classification of serial manipulators and computational techniques for manipulator algebra.

{\bf Acknowledgements}
The authors thank Dr J. M. Selig of London South Bank University, who first aroused their interest in invariants of screws, multi-screws and screw systems and who suggested the connection with Denavit--Hartenberg parameters. We also thank E.~Walker of Texas A\&M University for correcting an error in one of our SAGBI basis computations.


\end{document}